\documentclass[10pt]{amsart}
\usepackage[usenames]{color} 
\usepackage{amssymb} 
\usepackage{amsmath} 
\usepackage{hyperref}
\usepackage[utf8]{inputenc} 

\newcommand{\Z}{\mathbb{Z}}
\newcommand{\Q}{\mathbb{Q}}
\newcommand{\R}{\mathbb{R}}
\newcommand{\C}{\mathbb{C}}
\newcommand{\G}{\mathbb{G}}
\newcommand{\bbP}{\mathbb{P}}
\newcommand{\A}{\mathbb{A}}
\newcommand{\val}{\mathrm{val}}
\newcommand{\Berk}{\mathrm{Berk}}
\newcommand{\Trop}{\mathrm{Trop}}
\newcommand{\Gal}{\mathrm{Gal}}

\theoremstyle{plain}
\newtheorem{thm}{Theorem}[section]
\newtheorem{prop}[thm]{Proposition}
\newtheorem{lmm}[thm]{Lemma}
\newtheorem{cor}[thm]{Corollary}
\newtheorem{conj}[thm]{Conjecture}
\newtheorem{rk}[thm]{Remark}
\newtheorem{ex}[thm]{Example}
\newcommand{\poubelle}[1]{}

\begin{document}

\title[Lower bounds for heights ]{Lower bounds for heights on some algebraic dynamical systems}
	
	\author[A. Plessis]{Arnaud Plessis}
	\address[A. Plessis]{Yanqi Lake Beijing Institute of Mathematical Sciences and Applications, Huairou District, Beijing, China. ORCID number: 0000-0003-3678-6036}
	\email{plessis@bimsa.cn}
	
	\author[S. Sahoo]{Satyabrat Sahoo}
	\address[S. Sahoo]{Yau Mathematical Sciences Center (YMSC), Tsinghua University, Haidian  District, Beijing 100084, China. ORCID number: 0000-0002-5995-7208}
	\email{sahoos@mail.tsinghua.edu.cn}
	
	\date{\today}

\maketitle

\begin{abstract}
Let $v$ be a finite place of a number field $K$ and write $K^{nr,v}$ for the maximal field extension of $K$ in which $v$ is unramified. 
The purpose of this paper is split up into two parts. 
The first one generalizes a theorem of Pottmeyer: If $E$ is an elliptic curve defined over $K$ with split multiplicative reduction at $v$, then the N\'eron-Tate height of a non-torsion point $P\in E(\bar{K})$ is bounded from below by $C / e_v(P)^{2 e_v(P)+1}$, where $C>0$ is an absolute constant and $e_v(P)$ is the maximum of all ramification indices $e_w(K(P) \vert K)$ with $w\vert v$. 
Among other things, we refine this result by showing that given a simple abelian variety $A$ defined over $K$ that is degenerate at $v$, the N\'eron-Tate height of a non-torsion point $P\in A(\bar{K})$ is at least $C / \mathrm{lcm}_{w\vert v} \{e_w(K(P)\vert K)\}^2$, where $C>0$ is an absolute constant. We then give applications towards Lehmer's conjecture.
Next, we provide the first examples of polynomials $\phi\in K[X]$ of degree at least $2$ so that the canonical height $\hat{h}_\phi$ of any point in $\bbP^1(K^{nr,v})$ is either $0$ or bounded from below by an absolute positive constant. 
\end{abstract}

\section{Introduction}
As usual, we denote by $\bar{K}$ an algebraic closure of a field $K$. Throughout this article, \textit{algebraic dynamical system} refers to a data $(V/K, \mathcal{L}, \phi)$, where $V$ is a projective variety defined over a number field $K$, $\mathcal{L}$ is a line bundle in the Picard group of $V$ and $\phi$ is a $K$-endomorphism of $V$. 

Let $(V/K, \mathcal{L}, \phi)$ be an algebraic dynamical system. 
In a prominent paper, Call and Silverman constructed a (canonical) height $\hat{h}_{\mathcal{L}, \phi} : V(\bar{K}) \to \R$ if $\phi^*\mathcal{L} \simeq \mathcal{L}^{\otimes q}$ for some integer $q\geq 2$ \cite{CS93}. 
Concretely, it is the unique Weil height function $\hat{h}$ for $\mathcal{L}$ satisfying $\hat{h}\circ \phi = q \hat{h}$ on $V(\bar{K})$. 
When $\mathcal{L}$ is ample, the function $\hat{h}_{\mathcal{L}, \phi}$ is Galois invariant, non-negative, vanishes precisely at the set of preperiodic points under $\phi$ (a point in $V(\bar{K})$ is said to be preperiodic under $\phi$ if its forward orbit under $\phi$ is finite) and Northcott's theorem holds: the set of points in $V(\bar{K})$ with bounded degree and bounded height is finite. 
Assume from now that $\mathcal{L}$ is ample.

An important (and hard) problem in diophantine geometry/dynamical systems is to understand the points of $V(\bar{K})$ with small Call-Silverman height. For instance, the celebrated Yuan's equidistribution theorem \cite{Y08} roughly states that the Galois orbit of such a point is equidistributed around the (Berkovich) Julia set of $\phi$, unless some natural algebraic obstructions occur. 
For some algebraic dynamical systems, we also have deep conjectures providing informations on the non-preperiodic points of small Call-Silverman height, like, for instance, Bogomolov conjecture, Lang's conjecture or Lehmer's conjecture. 

Another approach consists in understanding the (algebraic) fields $L$ for which $V(L)$ has no point of small Call-Silverman height, except the preperiodic points.
Following Fili and Miner \cite{FM15}, we say that $V(L)$ has the \textit{Bogomolov property} relative to $\hat{h}_{\mathcal{L}, \phi}$ if there is a positive constant $C$ such that $\hat{h}_{\mathcal{L},\phi}(P)\geq C$ for all non-preperiodic points $P\in V(L)$. 
In addition, if $V(L)$ contains only finitely many preperiodic points under $\phi$, then $V(L)$ is said to have the \textit{strong Bogomolov property} relative to $\hat{h}_{\mathcal{L}, \phi}$. 
Northcott's theorem above ensures that $V(L)$ has the strong Bogomolov property relative to $\hat{h}_{\mathcal{L}, \phi}$ for all number fields $L$ containing $K$. 
Hence, this property is only relevant for fields of infinite degree.
Up to our knowledge, this approach is studied in only three cases so far:

\begin{enumerate} 
\item $V/K=\bbP^1/\Q, \mathcal{L}=\mathcal{O}(1)$, the dual of the tautological line bundle over $\bbP^1$, and $\phi(X)=X^2$. 
Hence, $\hat{h}_{\mathcal{L}, \phi}$ is the (absolute, logarithmic) Weil height.
 \item $V/K=A/K$ is an abelian variety, $\mathcal{L}$ is a symmetric ample line bundle, and $\phi=[2]$ is the multiplication-by-$2$ map.
In that case, $\hat{h}_{\mathcal{L}, \phi}$ is the so-called N\'eron-Tate height, which we simply denote by $\hat{h}_\mathcal{L}$ from now. 
\item $V/K=\bbP^1/K, \mathcal{L}=\mathcal{O}(1)$ and $\phi$ is any rational function of $K(X)$ with degree at least $2$. 
We then get a dynamical height and we put $\hat{h}_\phi = \hat{h}_{\mathcal{L}, \phi}$.
\end{enumerate}
(We identified the set of $K$-endomorphisms of $\bbP^1$ with $K(X)$). 
Clearly, the first case fits into the third one. 
But the Weil height offers properties which are easier to exploit than those of an arbitrary dynamical height, and so it makes sense to consider it as a separate case (for instance, the set of preperiodic points under $X^2$ is $0$ and the set of all roots of unity, and the decomposition of the Weil height into a sum of N\'eron local height functions is more explicit \cite[Theorem 2.3]{CS93}).

In the next three subsections, we will establish a (non-exhaustive) state-of-the-art (as well as our main results) of each one of these three cases.
Throughout this introduction, $K$ denotes an arbitrary number field. 

\subsection{Bogomolov property relative to the Weil height}
For short, we say that a field $L$ has the Bogomolov property if $\bbP^1(L)$ has the Bogomolov property relative to the Weil height. 
This notation was first introduced by Bombieri and Zannier \cite{BZ01}. 
This subsection mainly lists examples of fields with the Bogomolov property.

In 1973, Schinzel established that $\Q^{tr}$, the maximal totally real field extension of $\Q$,  has the Bogomolov property \cite{S73}; it is the first known example of an algebraic field of infinite degree satisfying this property.

Another important result due to Amoroso and Zannier claims that $K^{ab}$, the maximal abelian extension of $K$, has the Bogomolov property \cite{AZ00}. 
For other examples of fields having the Bogomolov property, see \cite{ADZ14, F22, G16, G15, H13, P19, P24, P15b}.

End this subsection by giving an example of an algebraic field not satisfying the Bogomolov property.
For any finite place $v$ of $K$, write $K^{nr,v}$ for the maximal field extension of $K$ in which $v$ is unramified.  
It is easy to see that this field does not have the Bogomolov property. 
Indeed, let $p\geq 2$ be a rational prime not lying under $v$ and choose a positive integer $n$. 
Kummer's theory asserts that $v$ is unramified in $K(p^{1/p^n})$, that is, $p^{1/p^n}\in K^{nr,v}$. 
The claim follows since the Weil height of $p^{1/p^n}$, which is $(\log p)/p^n$, goes to $0$ as $n\to +\infty$. 

\subsection{Lower bounds for the N\'eron-Tate height on abelian varieties} \label{ss-section 1.2}
We now consider the case $(2)$ above of which we keep the notation. 
The literature concerning the Bogomolov property in the abelian case is much less extensive than the toric one (apart from elliptic curves, the authors only know three results). 

An immediate consequence of an equidistribution theorem of Zhang states that $A(K\Q^{tr})$ has the strong Bogomolov property relative to $\hat{h}_\mathcal{L}$ \cite{Z98}. 
This is the first known example regarding the Bogomolov property on abelian varieties. 

The analogue of Amoroso and Zannier's theorem above was showed by Baker and Silverman  \cite{BS04}: $A(K^{ab})$ has the Bogomolov property relative to $\hat{h}_\mathcal{L}$. 
Unlike Zhang, their proof does not involve equidistribution, but is based on obtaining metric estimates, which is the other classical method to get lower bounds for the Weil and N\'eron-Tate heights. 

The last example, due to Gubler, shows that there are algebraic fields $L$, not having the Bogomolov property, such that $A(L)$ has the strong Bogomolov property relative to $\hat{h}_\mathcal{L}$.  
More precisely, if $A$ is totally degenerate at some finite place $v$ of $K$ (see Subsection \ref{ss-section 2.2} for the definition), then $A(K^{nr,v})$ has the strong Bogomolov property relative to $\hat{h}_\mathcal{L}$ \cite[Corollary 6.7]{G07b}. 
Like Zhang, this result arises from an equidistribution theorem \cite[Corollary 6.6]{G07b}.
For more examples about the Bogomolov property in the elliptic case, we refer the reader to \cite{BP05, H13, P22, P24}. 

 A decade after Gubler, Pottmeyer stated a more precise claim in the elliptic case. 
The proof relies on a discrepancy computation. 
Given a finite extension $L/K$ and a finite place $w$ of $L$, we put $e_w(L\vert K)$ to be the ramification index of $w$ in $L/K$. 

\begin{thm} [Pottmeyer, \cite{P15}, Theorem 4.1]   \label{thm Pottmeyer}
Let $E$ be an elliptic curve defined over a number field $K$ with split multiplicative reduction at some finite place $v$ of $K$. 
Let $\mathcal{L}$ be a symmetric ample line bundle on $E$. 
Then there is an effective computable constant $C>0$ such that 
\begin{equation} \label{eqPot}
 \hat{h}_\mathcal{L}(P) \geq \frac{ C \cdot e^{2e_v(P)}}{e_v(P)^{2e_v(P)+1}}
 \end{equation}
  for all non-torsion points $P\in E(\bar{K})$, where $e_v(P)$  denotes the maximum of all ramification indices $e_w(K(P)\vert K)$ with $w$ a place of $K(P)$ extending $v$. 
\end{thm}  

(Note that we used Stirling equivalent in the formulation of this theorem). 
The lower bound in \eqref{eqPot} is super-exponential in $e_v(P)$. 

Let $A$ be an abelian variety defined over $K$.
For any field extension $L/K$, write $A_L$ for the base change of $A$ to $L$. 
Given a finite place $v$ of $\bar{K}$, we say that $A_{\bar{K}}$ is \textit{degenerate} at $v$ if $A_{\bar{K}}$ does not have good reduction at $v$. 
Next, an abelian variety defined over $\bar{K}$ is said to be \textit{simple} if $\{0\}$ is its only proper abelian subvariety. 

Among other things, the first aim of this text is to extend Theorem \ref{thm Pottmeyer} to simple abelian varieties that are degenerate at some finite place of $\bar{K}$.
In addition, the lower bound that we obtain is at most exponential in $e_v(P)$, so much sharper.
The least common multiple of a finite number of positive integers $a_1,\dots,a_n$ is denoted by $\mathrm{lcm}\{a_1,\dots,a_n\}$.

\begin{thm} \label{thm Abel Var}
Let $A$ be an abelian variety defined over a number field $K$, let $v$ be a finite place of $K$, and let $\mathcal{L}$ be a symmetric ample line bundle on $A$. 
Assume that for every simple abelian subvariety $B$ of $A_{\overline{K}}$, there exists a place $v_B$ of $\overline{K}$ such that $B$ is degenerate at $v_B$. Then there is an absolute constant $C>0$ such that 
 \begin{equation} \label{eqthm1.2}
 \hat{h}_{\mathcal{L}} (P) \geq C \cdot l_v(P)^{-2}
 \end{equation}
 for all non-torsion points $P \in A(\bar{K})$. Here, $l_v(P)=\mathrm{lcm}_{w|v}\{e_w(K(P)|K)\}$, where $w$ runs over all places of $K(P)$ lying over $v$.
\end{thm}

\begin{rk} \rm{The inequality $\mathrm{lcm}\{1,2,\dots,e_v(P)\}\leq 3^{e_v(P)}$ \cite{H72} yields $\hat{h}_\mathcal{L}(P)\geq C 9^{-e_v(P)}$ for all non-torsion points $P\in A(\bar{K})$, and so the lower bound in \eqref{eqthm1.2} is at most exponential in $e_v(P)$. 
Note that the inequality $l_v(P) \leq 3^{e_v(P)}$ is far from being precise since $e_w(K(P)\vert K)$ cannot always freely run the range $\{1,\dots,e_v(P)\}$ (for instance, their sum must be less or equal than the degree of $P$ over $K$). }
\end{rk}

\begin{rk} \rm{Unlike Theorem \ref{thm Pottmeyer}, the constant $C$ in \eqref{eqthm1.2} is not explicit. 
The reason is that we use a tropical equidistribution theorem, which is not qualitative, rather than an explicit discrepancy computation. 
In fact, the proof of \eqref{eqthm1.2} only relies on tropical arguments and we state in Theorem \ref{thm trop} a tropical version of Theorem \ref{thm Abel Var}.}
\end{rk}

\begin{rk} \rm{In particular, $A(K^{nr,v})$ has the Bogomolov property relative to $\hat{h}_\mathcal{L}$ since it is the set of points $P \in A(\bar{K})$ for which $e_v(P)=1$. 
If $A$ is totally degenerate at $v$, then $A_{\bar{K}}$ is totally degenerate at any place of $\bar{K}$ lying over $v$, and so all abelian subvarieties of $A_{\overline{K}}$ are totally degenerate at any place of $\bar{K}$ lying over $v$ \cite[Lemma 6.1(a)]{G07b}. 
Theorem \ref{thm Abel Var} is therefore a generalization of Gubler's theorem above, and it is even stronger since there exist simple abelian varieties that are degenerate at some finite place of $\bar{K}$, but not totally degenerate. 
Also, we can apply Theorem \ref{thm Abel Var} to elliptic curves with non-split multiplicative reduction at $v$.  }
\end{rk}

\begin{rk} \label{rk 1.10}  \rm{The converse of Theorem \ref{thm Abel Var} is true up to a finite extension of $K$. 
The proof closely follows that of Pottmeyer in the elliptic case \cite[Proposition 5.6]{P15}. 
Indeed, suppose that there is a simple abelian subvariety $B$ of $A_{\bar{K}}$ such that $B$ has good reduction at any place of $\bar{K}$ extending $v$. 
Choose a number field of definition $L$ of $B$ for which the latter has good reduction at all places of $L$ lying over $v$ and such that $B(L)$ contains a non-torsion point. 
Let $P\in B(L)$ be a non-torsion point, and let $p\geq 2$ be a rational prime not lying under $v$. 
For every $n\geq 1$, denote by $P_n$ any point of $B(\overline{K})$ satisfying $[p^n]P_n=P$. 
Let $w$ be a place of $L$ extending $v$. 
The Chevalley-Weil theorem \cite[Proposition 10.3.10]{BG06} asserts that $w$ is unramified in $L(P_n)$, that is, $P_n\in B(L^{nr,w})$. 
It becomes clear that $P_n\in B\left(L^{nr,w} \right)$ is a non-torsion point such that $\hat{h}_{\mathcal{L}}(P_n)=\hat{h}(P)/p^{2n}\to 0$.}
\end{rk}

The lower bound in \eqref{eqthm1.2} is precise enough to find new applications towards Lehmer's conjecture below, see \cite[Conjecture 0.2]{BS04} or \cite{M84}.

\begin{conj} [Lehmer's conjecture] \label{Lehmer conjecture} 
Let $A$ be an abelian variety defined over a number field $K$, and let $\mathcal{L}$ be a symmetric ample line bundle on $A$. 
Then there is a constant $C>0$ such that
				$$\hat{h}_{\mathcal{L}}(P) \geq C D^{-1/g_0(P)}$$ for all non-torsion points $P \in A(\bar{K})$, where $D$ denotes the degree of the extension $K(P)/K$ and $g_0(P)$ denotes	the dimension of the smallest algebraic subgroup of $A$ containing $P$.
\end{conj}
 
 Let $P\in A(\bar{K})$ be a non-torsion point of degree $D$ over $K$. 
 All constants $C$ stated below are positive and independent of $P$. 
 Set $g$ to be the dimension of $A$.

 In \cite{DH00}, David and Hindry proved that if $A$ has complex multiplication (CM) and if $g_0(P)=g$, then for all $\varepsilon >0$, there is $C_\varepsilon>0$ such that $\hat{h}_\mathcal{L}(P)\geq C_\varepsilon D^{-(1/g) -\varepsilon}$.
In \cite{M84}, Masser proved that $\hat{h}_{\mathcal{L}}(P) \geq C D^{-(2g+6+(2/g))}$. Later in \cite{M86}, Masser improved this bound to $ C D^{-(2g+1)} (\log{D})^{-2g}$ (and even $C D^{-2} (\log D)^{-1}$ if $A$ is CM). 
 
If $A=E$ is an elliptic curve, David proved that $\hat{h}_\mathcal{L}(P) \geq C D^{-15/8} (\log(2D))^{-2}$ provided that the $j$-invariant of $E$ is not an algebraic integer \cite{D97}. 

In the esteemed work of Amoroso and Masser \cite{AM16}, the authors first studied Lehmer-type bounds for any $\alpha \in \bar{\Q}^\times$ of infinite order such that $\Q(\alpha)$ is a Galois extension of $\Q$. 
In \cite{GM17}, Galateau and Mah\'e proved that $\hat{h}_\mathcal{L}(P) \geq C D^{-1}$ if $A=E$ is an elliptic curve and if $K(P)/K$ is Galois. In \cite{KS24}, Kumar and the second author proved that $\hat{h}_{\mathcal{L}}(P) \geq C (D\log D)^{-2g}$ if $K(P)/K$ is Galois.

Let $v$ be a finite place of $K$. 
It is well-known that $e_w(K(P) \vert K)$ is independent of the place $w$ of $K(P)$ lying over $v$ if $K(P)/K$ is Galois. 
As a corollary of Theorem~\ref{thm Abel Var}, we have the following result.

\begin{cor} \label{Leh type low bound for abel var}
	Keep the same notations and assumptions as in Theorem~\ref{thm Abel Var}. 
	Then there exists an absolute constant $C>0$ such that 
	\begin{equation*}
		\hat{h}_{\mathcal{L}} (P) \geq C \cdot e(K(P)|K)^{-2} \geq C D^{-2},
	\end{equation*}
for all non-torsion points $P \in A(\bar{K})$ such that $K(P)/K$ is Galois of degree $D$, where $e(K(P)|K)$ is the ramification index of any place of $K(P)$ extending $v$. 
\end{cor}
   \begin{rk} \rm{The lower bound $CD^{-2}$ in Corollary~\ref{Leh type low bound for abel var} improves the recent Lehmer-type lower bound $ C (D\log D)^{-2g}$ in \cite{KS24} stated above, at least for abelian varieties considered in Theorem \ref{thm Abel Var}. 
Note that it is independent of the dimension $g$ of $A$. 
 To get the lower bound $CD^{-2}$ for all abelian varieties, it remains to establish it for all geometrically simple abelian varieties with good reduction everywhere (the proof is an easy adaptation of arguments presented in Subsection \ref{ss-section 3.2}, which are mainly based on the Poincar\'e's reducibility theorem). }
   \end{rk}
\begin{rk}	\rm{We only used the tropical geometry to obtain the Lehmer-type bound in Corollary~\ref{Leh type low bound for abel var}, which seems to be a novel approach. Consequently, it may be worthwhile to investigate this innovative concept to ascertain whether the tropicalization of closed subvarieties of an abelian variety could provide deeper results towards Lehmer's conjecture.}
\end{rk}
  	
\subsection{Bogomolov property relative to a dynamical height}

Given an algebraic extension $L/K$ and a rational function $\phi\in K(X)$ of degree at least two, it is a hard task to know whether  $\bbP^1(L)$ has the Bogomolov property relative to $\hat{h}_\phi$. 
In view of the first two subsections, it is natural to consider the case $K=\Q$ and $L=\Q^{tr}$ first, which was solved by Pottmeyer \cite{P13}: $\bbP^1(\Q^{tr})$ has the Bogomolov property relative to $\hat{h}_\phi$ if and only if $\Q^{tr}$ has only finitely many preperiodic points under $\phi$. 
 
 Again, the previous subsections naturally lead to now handling the case $L=K^{ab}$.
 In \cite{P13}, Pottmeyer also proved that $\bbP^1(K^{ab})$ has the Bogomolov property relative to $\hat{h}_\phi$ when $\phi$ is a Chebyshev polynomial. 
 Very recently, Looper got the same conclusion when $\phi$ is a polynomial with bad reduction at some finite place $v$ of $K$, and with a superattracting finite periodic point \cite{L21} (unpublished). 
 Her proof (mainly) combines a discrepancy computation with an equidistribution theorem that she proved in a remarkable previous paper \cite{L21b}.  
 Note that so far, we do not know whether there is a rational function $\phi\in K(X)$ of degree at least two so that $\bbP^1(K^{ab})$ does not have the Bogomolov property with respect to $\hat{h}_\phi$.
 
 Finally, what about the third "classical" case, namely $L=K^{nr,v}$, where $v$ is a finite place of $K$? 
We saw that this field does not have the Bogomolov property and that $\bbP^1(K^{nr,v})$ has the Bogomolov property relative to $\hat{h}_\phi$ if $\phi$ is a Latt\`es map of an elliptic curve $E$ with split multiplicative reduction at $v$ \cite{P15} (it is a direct consequence of Theorem \ref{thm Pottmeyer} since if $\alpha\in\bar{K}$, then $\hat{h}_\phi(\alpha)=2\hat{h}_\mathcal{L}(P)$ for any $P\in E(\bar{K})$ with abscissa $\alpha$, where $\mathcal{L}$ is the unique ample generator of the N\'eron-Severi group of $E$). 
 Up to our knowledge, nothing further is known about this case. 
 
The second goal of this paper is to provide much more examples.  
  The most general result that we will show is Theorem \ref{thm general}, which can be rephrased as follows: If the (Berkovich) Julia set of $\phi$ has an "atypical" element, then $\bbP^1(K^{nr,v})$ has the strong Bogomolov property relative to $\hat{h}_\phi$. 
However, the Julia set of $\phi$ is very abstract and in practice, it is not trivial to determine whether such an element exists (actually, it does not if, for instance, $\phi$ has good reduction at $v$, see Remark \ref{rk 3.1}). 
Nevertheless, if the Newton polygon of the polynomial $\phi(X)-X\in K_v[X]$ has some properties, then it is possible to exhibit an "atypical" element in the Julia set of $\phi$, which therefore proves that $\bbP^1(K^{nr,v})$ has the strong Bogomolov property relative to $\hat{h}_\phi$.
The connection between Newton polygons and the Julia set of $\phi$ will be explicitly stated in Subsection \ref{ss-section 4.4}. 

Let $\psi(X)=\sum_{i=0}^d a_i X^i\in K_v[X]$ be a polynomial with coefficients in $K_v$ such that $a_da_0\neq 0$. 
The Newton polygon of $\psi$ is defined to be the lower boundary of the convex hull of the set of points $(i, -\log \vert a_i\vert_v)$, ignoring the points with $a_i=0$, where $\vert . \vert_v$ denotes the normalized $v$-adic absolute value on $K_v$.  
  
\begin{thm} \label{thm Berk}
Choose a finite place $v$ of a number field $K$.
Let $\phi \in K[X]$ be a polynomial of degree $d\geq 2$ with leading coefficient $a_d$ such that $\phi(0)$ is non-zero. 
Denote by $\mu_1,\dots, \mu_r$ the slopes of the line segments of the Newton polygon of $\phi(X)-X$, viewed as a polynomial over $K_v$.  
If there is $l\in\{1,\dots,r\}$ satisfying \[ \mu_l\notin \frac{\log p}{e_v(K\vert\Q)}  \Z \; \; \;  \text{and} \; \; \; \mu_l\geq -\frac{\log \vert a_d\vert_v}{d-1},\]  then $\bbP^1(K^{nr,v})$ has the strong Bogomolov property relative to $\hat{h}_\phi$. 
\end{thm}

\begin{ex} \rm{Let $p\geq 2$ be a rational prime and take a polynomial $\phi(X) = a_dX^d +\dots+a_2X^2 + X +a_0 \in \Q[X]$ of degree $d\geq 2$ such that $\vert a_0 \vert_p=p$ and $\vert a_i\vert_p=1$ for all $i\geq 2$. 
The Newton polygon of $\phi(X)-X$, viewed as a polynomial in $\Q_p[X]$, is made of two line segments. The first one joins points $(0,-\log p)$ and $(2,0)$ while the second one connects points $(2,0)$ and $(d,0)$. 
The slope of the first line segment is $(\log p)/2$, which clearly satisfies the conditions of the theorem. 
Thus, $\bbP^1(\Q^{nr,p})$ has the strong Bogomolov Property relative to $\hat{h}_\phi$. }
\end{ex}

\subsection*{Acknowledgement} It is a real pleasure to thank Gao, Gubler, Poineau, Silverman and Yamaki for nicely answering our questions. 
We also thank an anonymous referee for pointing out a mistake in a previous formulation of Lehmer's conjecture. 

\section{Analytification and tropicalization}

We fix throughout this section an algebraically closed field $\mathbb{K}$ which is complete with respect to a non-trivial non-archimedean absolute value $\vert .\vert$. 
The aim of this section is to briefly explain the concepts of analytification and tropicalization.
For simplicity's sake, we only expose the cases that we will use in the sequel of the text. 

\subsection{Analytification} \label{ss-section 2.1}
Analytification is roughly the process of universally turning a "bad" algebraic space into a "nice" analytic space. 
Formally, let $X$ be an affine variety defined over $\mathbb{K}$, that is, $X = \mathrm{Spec} R$, where $R$ is a commutative $\mathbb{K}$-algebra of finite type without nilpotent elements. 
In that case, the analytification of $X$, denoted by $X^{an}$, is the set of multiplicative seminorms on $R$ extending the absolute value on $\mathbb{K}$. 
In other words, $X^{an}$ is the set of maps $\zeta : R \to \R_{\geq 0}$ such that 
\begin{enumerate}
\item (extension) $\zeta(a)= \vert a\vert$ for all $a\in\mathbb{K}$; 
\item (multiplicativity) $\zeta(PQ) = \zeta(P)\zeta(Q)$ for all $P,Q\in R$; 
\item (triangle inequality) $\zeta(P+Q) \leq \zeta(P) + \zeta(Q)$ for all $P,Q\in R$. 
\end{enumerate}
The set $X^{an}$ is called the \textit{Berkovich analytic space} associated to $X$. 
We endow $X^{an}$ with the Berkovich topology, that is, the coarsest one for which the maps $\zeta \mapsto \zeta(P)$ are continuous for every $P\in R$. 
The special case $R=\C_v[x]$ will be studied in more detail in Subsection \ref{ss-section 4.1} (note that $X$ is then the affine line over $\C_v$). 

If $X$ is now assumed to be a projective variety defined over $\mathbb{K}$, then it has a finite open affine covering $\{U_i\}_{i=1}^n$ and we can define the analytification $X^{an}$ of $X$ by glueing the spaces $U_1^{an}, \dots, U_n^{an}$. 

Whatever if $X$ is affine or projective, the set of its $\mathbb{K}$-rational points embeds into $X^{an}$ and its image is dense in $X^{an}$. 
In addition, $X^{an}$ is Hausdorff (the singletons are therefore closed) and $X^{an}$ is arcwise connected if and only if $X$ is connected. 

The general construction of Berkovich analytic spaces can be found in \cite[Appendix C]{BR10} or \cite{B90}. 
Here we followed what was explained by Gubler in \cite[\S 2]{GU13}.

\begin{ex} \label{ex 2.1}  \rm{Let $A$ be an abelian variety defined over $\mathbb{K}$. 
There is an analytic group $E$ as well as a surjective homomorphism $E \to A^{an}$ of analytic groups whose kernel $M$ is a lattice in $E(\mathbb{K})$. In other words, $A^{an} \simeq E/M$. 
For more details, we refer the reader to \cite[\S 4]{GU10} or \cite[\S 3.1]{YA}. }
\end{ex}

\subsection{Tropicalization} \label{ss-section 2.2}
Tropicalization is roughly the process of transferring algebraic geometry into convex geometry. 
Let $\G_m^n$ denote the multiplicative group of dimension $n$ over $\mathbb{K}$. 
We fix coordinates $x_1,\dots,x_n$ of $\G_m^n$ so that $\G_m^n = \mathrm{Spec} \; \mathbb{K}[x_1^{\pm}, \dots, x_n^{\pm}]$. 
By the previous subsection, every element $\zeta \in (\G_m^n)^{an}$ is a multiplicative seminorm on $\mathbb{K}[x_1^{\pm}, \dots, x_n^{\pm}]$ extending $\vert .\vert$ on $\mathbb{K}$. 
In particular, $\zeta(x_j)\neq 0$ for all $j\in\{1,\dots,n\}$ since $1= \zeta(1)=\zeta(x_j/x_j)= \zeta(x_j) \zeta(x_j^{-1})$. 

Let $A$ be an abelian variety defined over $\mathbb{K}$, and let $E$ and $M$ be as in Example \ref{ex 2.1}. 
There are a unique non-negative integer $n$ as well as a unique abelian variety up to isomorphism such that 
\begin{equation} \label{eq1}
 1 \longrightarrow (\G_m^n)^{an} \longrightarrow E \xrightarrow{q^{an}} B \longrightarrow 0
 \end{equation} 
is a short exact sequence (the so-called Raynaud extension of $A$). 
The integer $n$ is called the \textit{torus rank} of $A$. 
We say that $A$ has \textit{good reduction} if $n=0$, is \textit{degenerate} if $n>0$ and is \textit{totally degenerate} if $n$ equals the dimension of $A$. 
The last three definitions coincide with those introduced in Subsection \ref{ss-section 1.2} when $\mathbb{K}$ is the completion of $\overline{K}$ with respect to $v$. 

The exact sequence \eqref{eq1} is locally trivial over $B$, that is, there exists a  formal covering $\{V_i\}_{i\in I}$ of $B$ such that $(q^{an})^{-1}(V_i) \simeq V_i \times (\G_m^n)^{an}$ as $(\G_m^n)^{an}$-torsors for all $i\in I$. 
Composing this isomorphism with the second projection gives a map $r_i : (q^{an})^{-1}(V_i) \to (\G_m^n)^{an}$. 
Recall that $\zeta(x_j)\neq 0$ for all $j\in\{1,\dots,n\}$ and all $\zeta\in (\G_m^n)^{an}$. 
We thus get a well-defined map  \[\begin{array}{ccccc}
 &  & (q^{an})^{-1}(V_i) & \to & \R^n \\
 & & z & \mapsto & (-\log r_i(z)(x_1),\dots, -\log r_i(z)(x_n) ) \\
\end{array}.\]
These morphisms patch together to form a continuous surjective group homomorphism $\mathrm{val} : E \to \R^n$. 
The lattice $M$ maps to a lattice $\Lambda=\mathrm{val}(M)$ in $\R^n$. 
Example \ref{ex 2.1} then provides a continuous surjective group homomorphism $\overline{\mathrm{val}} : A^{an} \to \R^n/\Lambda$.
The subgroup $\overline{\val}^{-1}(\overline{\mathbf{0}})$ of $A^{an}$ is an analytic subgroup.

Finally, the tropicalisation of a closed subvariety $X$ of $A$ is defined as the set $\Trop(X)=\overline{\val}(X^{an})$. 
If $X$ is connected, then $X^{an}$ is also connected by Subsection \ref{ss-section 2.1}. 
The direct image of a connected space under a continuous function being connected, we infer that $\Trop(X)$ is connected. 
For more details concerning this subsection, the reader can consult \cite[\S 4.2]{GU10} or \cite[\S 3.1]{YA}.

\section{Tropical version of Theorem \ref{thm Abel Var}}
We fix once and for all a number field $K$ as well as a finite place $v$ of $K$. 
For any algebraic extension $L/K$ and any finite place $w$ of $L$, we denote by $L_w$ the $w$-adic completion of $L$. 
It is well-known that the normalized absolute value $\vert . \vert_w$ on $L_w$ uniquely extends to an absolute value on $\overline{L_w}$. 
We now set $\C_w$ to be the completion of $\overline{L_w}$ with respect to this absolute value. 
We call again $\vert . \vert_w$ for the unique absolute value on $\C_w$ extending $\vert . \vert_w$ on $L_w$. 
Note that $(\overline{K})_w$ is complete (by definition) and algebraically closed (a consequence of Krasner's lemma), and so $(\overline{K})_w=\C_w$. 
 
Next, given an algebraic variety $X$ defined over $L$, we simply denote by $X_w$ the base change of $X$ to $\C_w$. 
For any abelian variety $A$ and any integer $m$, we denote by $[m] : A \rightarrow A$ the multiplication-by-$m$ map.

The goal of this section is to prove the tropical version of Theorem \ref{thm Abel Var} below. 

\begin{thm} \label{thm trop}
Let $A$ be an abelian variety defined over $K$, and let $\mathcal{L}$ be a symmetric ample line bundle on $A$. 
Assume that for all simple abelian subvarieties $B$ of $A_{\overline{K}}$, there is a finite place $v_B$ of $\overline{K}$ extending $v$ such that $\Trop(X_{v_B})$ has cardinality at least two for all irreducible closed subvarieties $X$ of $B$ with positive dimension.
Then there is a constant $C>0$ such that \[\hat{h}_{\mathcal{L}} (P) \geq C \cdot l_v(P)^{-2}\] 
 for all non-torsion points $P \in A(\bar{K})$. Here, $l_v(P)=\mathrm{lcm}_{w|v}\{e_w(K(P)|K)\}$, where $w$ runs over all places of $K(P)$ lying over $v$.
\end{thm}

The bridge between the arithmetic nature of Theorem \ref{thm Abel Var} and the geometric nature of Theorem \ref{thm trop} is ensured by a lemma of Yamaki, which we will prove in the first subsection of this paragraph.
The second subsection shows the theorem when $A$ is geometrically simple, that is, $A_{\overline{K}}$ is a simple abelian variety. 
Finally, we will show Theorem \ref{thm trop} in full generality thanks to Poincar\'e' s reducibility theorem.

The main idea of the proof consists in contradicting the tropical equidistribution theorem.
If the conclusion of Theorem \ref{thm trop} is false, then we can find a sequence $(P_m)_m$ such that $\hat{h}_\mathcal{L}(Q_m)\to 0$, where $Q_m= [ l_v(P_m)] P_m$. 
Up to an infinite subsequence, the tropical equidistribution theorem claims that the sequence of terms $S_m=\overline{\val}(\Gal(\bar{K}/K). Q_m)$ is equidistributed around $\Trop(X_{v_B})$, where $X$ is some irreducible closed subvariety of positive dimension of the simple abelian variety $B=A_{\overline{K}}$.
We will then show that the cardinality of $S_m$ cannot exceed some fixed integer, which contradicts the fact that $\Trop(X_{v_B})$ is a connected set of cardinality at least two, and so infinite. 

\subsection{A lemma due to Yamaki}
Let $B$ be a simple abelian variety defined over $\overline{K}$, which is degenerate at some finite place $w$ of $\overline{K}$, and let $X$ be an irreducible closed subvariety of $B$.
In \cite[Lemma 7.15]{YA}, Yamaki proved that the tropicalization of $X_w$ is one point if and only if $X$ is one point.
However, he proved it when $K$ is a function field, but the proof remains exactly the same in the number field case. 
As it is short, we decided to reproduce it for the convenience of the reader. 

\begin{lmm} [Yamaki's lemma] \label{Yamaki}
Let $B$ be a simple abelian variety defined over $\overline{K}$ and assume that $B$ is degenerate at some finite place $w$ of $\overline{K}$.  
Let $X$ be an irreducible closed subvariety of $B$. 
Then $\Trop(X_w)$ consists to a single point if and only if $X$ is a singleton. 
\end{lmm}

\begin{proof}
The converse is clear since if $X$ is a singleton, then $(X_w)^{an}=X_w$ (the singleton $X_w$ is closed and dense in $(X_w)^{an}$ by Subsection \ref{ss-section 2.1}), and therefore $\Trop(X_w)=\overline{\val}(X_w)$ consists to a single point.  

Let us show the forward direction. 
Let $P\in X(\overline{K})\subset X^{an}$. 
As $\Trop(X_w)$ has cardinality one by assumption, we infer that it is equal to $\overline{\val}(P)$. 
Now, consider the irreducible closed subvariety $Y=X-P$ of $B$. 
By Subsection \ref{ss-section 2.2}, the map $\overline{\val}$ is a group homomorphism. 
Hence $\Trop(Y_w)=\overline{\val}(X_w-P)= \overline{\val}(X_w)-\overline{\val}(P)=\{\overline{\mathbf{0}}\}$. 
In other words, by replacing $X$ with $Y$ if needed, we can assume without loss of generality that $0\in X$, and so $\Trop(X_w)= \{\overline{\mathbf{0}}\}$.

Put $\langle X\rangle$ to be the smallest abelian subvariety of $B$ containing $X$. 
Let us consider, for each positive integer $l$, a morphism \[ \begin{array}{ccccc} 
 &  & X^{2l} & \to & B \\ 
& & (x_1,x_2,\dots,x_{2l-1}, x_{2l}) & \mapsto & \sum_{i=1}^l x_{2i-1} - x_{2i} \\ 
\end{array}\]  
and write $X_l$ for its image. It is an irreducible closed subvariety of $B$, which is contained in $\langle X\rangle$ since the latter contains all linear combinations of points in $X$. 
Taking $x_2=0$ leads to $X\subset X_1$. Similarly, take $x_{2l-1}=x_{2l}=0$ yields $X_{l-1}\subset X_l$. 
Since each $X_l$ is irreducible, there is an integer $j$ such that $X_l = X_j$ for all $l\geq j$. 
By definition, we have $X_l+X_m \subset X_{l+m}$, $0\in X_l$ and $-X_l=X_l$ for all positive integers $l$ and $m$. 
Consequently, $\bigcup_{l\geq 1} X_l = X_j$ is an algebraic subgroup of the abelian variety $\langle X\rangle$. 
Since $X_j$ is irreducible, it follows that it is an abelian variety. 
Finally, $\langle X \rangle = X_j$ since $\langle X\rangle$ is the smallest abelian subvariety of $B$ containing $X$. 

Let $G$ be an analytic subgroup of $(B_w)^{an}$ containing $X_w$. 
Thus,  it contains all linear combinations of points in $X_w$; whence $(X_j)_w\subset G$. 
Hence, $\langle X \rangle_w=(X_j)_w$ is the smallest analytic subgroup of $(B_w)^{an}$ containing $X_w$. 

We have $\Trop(X_w)= \{\overline{\mathbf{0}}\}$ by assumption; whence $X_w \subset \overline{\val}^{-1}(\overline{\mathbf{0}})$. 
Define $n$ as the torus rank of $B_w$. 
We have $n\geq 1$ since $B_w$ is degenerate by assumption. 
Subsection \ref{ss-section 2.2} asserts that the map $\overline{\val} : B_w^{an} \to \R^n / \Lambda$ is continuous, surjective and $\overline{\val}^{-1}(\overline{\mathbf{0}})$ is an analytic subgroup of $(B_w)^{an}$. 
Thus, $\overline{\val}^{-1}(\overline{\mathbf{0}})$ is a proper closed set in $(B_w)^{an}$.
The analytic group $B_w$ is dense in $(B_w)^{an}$ by Subsection \ref{ss-section 2.1}.
From all this, we infer that $B_w$ is not included in $\overline{\val}^{-1}(\overline{\mathbf{0}})$. 
Finally, $B_w$ and  $\overline{\val}^{-1}(\overline{\mathbf{0}})$ are analytic groups containing $X_w$, and so $\langle X \rangle_w\subset \overline{\val}^{-1}(\overline{\mathbf{0}}) \cap B_w \subsetneq B_w$. 
As $B_w$ is simple, we conclude that $\langle X\rangle_w =\{0\}$, which ends the proof of the lemma. 
\end{proof}

\textbf{Proof of Theorem \ref{thm Abel Var} by assuming Theorem \ref{thm trop}:}  
Let $B$ be a simple abelian subvariety of $A_{\overline{K}}$. 
By assumption, $B$ is degenerate at $v_B$. 
Yamaki's lemma claims that $\Trop(X_{v_B})$ has cardinality at least two for all irreducible closed subvarieties $X$ of $B$ with positive dimension. 
We now end the proof by using Theorem \ref{thm trop}. 
\qed 

\subsection{Proof of Theorem \ref{thm trop}: A special case} \label{Section 3}
We assume in this subsection that $B=A_{\bar{K}}$ is a simple abelian variety. 
For brevity, set $\nu=v_B$.  
We now fix the field embedding $\bar{K} \hookrightarrow \overline{K_v}$ associated to $\nu$ and we now see $\bar{K}$ as a subfield of $\overline{K_v}$. 

Example \ref{ex 2.1} provides a uniformization $A_v^{an}= E/M$ for some analytic group $E$ and some lattice $M$ in $E(\C_v)$.
The abelian variety $A$ is defined over $K$, and so over $K_v$. 
Thus, \cite[\S1]{BL91} asserts the existence of a finite extension $F/K_v$ for which the equality $A(L)= E(L)/M$ is true for all algebraic extensions $L/F$. 
In particular, $M\subset E(F)$. 
Moreover, this reference also claims that the map $\val$ defined in Subsection \ref{ss-section 2.2} maps $E(L)$ to $(\log \vert L^\times \vert_v)^n$, where $\vert L^\times\vert_v$ denotes the value group of $L$.
Hence, $\Lambda=\val(M) \subset (\log \vert F^\times\vert_v)^n$ and $\overline{\val}$ maps $A(L)$ to $(\log \vert L^\times \vert_v)^n/\Lambda$.
From now, set $p\geq 2$ to be the rational prime lying under the fixed place $v$. 

\begin{lmm}
	\label{int. lemma1}
	The set $S=\{\overline{\mathrm{val}}([l_v(P)]P), P\in A(\bar{K})\}$ is finite.
	\end{lmm}
\begin{proof}
Take $P\in A(\bar{K})$. 
Then $\overline{\val}(P)=(\log r_1, \dots, \log r_n)+\Lambda$ for some $r_i$ in $\vert F(P)^\times\vert_v = p^{\Z/e(F(P) \vert \Q_p)}$, where $e(L'\vert F')$ denotes the ramification index of a finite extension of local fields $L'/F'$.
As the map $\overline{\val}$ is a group homomorphism, we get
	\begin{equation} \label{eq 2}
	\overline{\val}([l_v(P)]P)= l_v(P) \cdot \overline{\val}(P) = \left(s_1,\dots,s_n\right)+ \Lambda
\end{equation}	
with $s_i= \log (r_i^{l_v(P)})$ for all $i$. 
Note that $s_i \in l_v(P) \cdot \log p \cdot  (\Z/e(F(P)| \Q_p))$. 

	Let $\nu_P$ be the place of $K(P)$ associated to the fixed embedding $\bar{K} \hookrightarrow \overline{K_v}$. 
Since $l_v(P)$ is a multiple of $e_{\nu_P}( K(P) | K) = e(K_v(P) | K_v)$, we deduce that \[s_i \in \frac{\log p}{e(F(P)\vert K_v(P))e(K_v\vert \Q_p)} \cdot \Z.\]
As $e(F(P)\vert K_v(P))\leq e(F\vert K_v)$, and so $e(F(P)\vert K_v(P))$ divides $e(F\vert K_v)!$, we get $s_i \in \log p \cdot (\Z/e)$, where $e=e(F| K_v)!e(K_v| \Q_p)$, which is independent of $P$. 
The finiteness of $S$ now arises from \eqref{eq 2} since $\Lambda$ is a lattice defined over $ \log p \cdot (\Z/ e(F\vert \Q_p))$.
\end{proof}

	\subsection*{Proof of Theorem~\ref{thm trop}:}
	Assume by contradiction that for all integers $m\geq 1$, there exists a non-torsion point $P_m \in B(\bar{K})= A(\bar{K})$ such that 
	\[ \hat{h}_{\mathcal{L}} (P_m) < \frac{1}{m \cdot l_v(P_m)^2}.\]	For each $m$, put  $Q_m= [l_v(P_m)]P_m$. Then $\hat{h}_\mathcal{L} (Q_m) \rightarrow 0$ as $m \rightarrow \infty$. 
	
There are a subsequence $(Q_{\phi(m)})_m$ of $(Q_m)_m$ and an irreducible closed subvariety $X$ of $B$ such that $Q_{\phi(m)}\in X(\bar{K})$ for all $m$ and such that no infinite subsequence of $(Q_{\phi(m)})_m$ is contained in a proper closed subvariety of $X$. 
Indeed, if no infinite subsequence of $(Q_m)_m$ is contained in a proper closed subvariety of $B$, then we take $Q_{\phi(m)}=Q_m$ and $X=B$.  
Otherwise, since any projective variety is a finite union of irreducible closed subvarieties, the pigeonhole principle provides an infinite subsequence $(Q_{\psi(m)})_m$ of $(Q_m)_m$ and an irreducible closed subvariety $Y\subsetneq B$ such that $Q_{\psi(m)}\in Y(\bar{K})$ for all $m$. 
The dimension of $Y$ being less than that of $B$, we can then repeat this argument for at most finitely many times and the claim follows.

Let $m$ be an integer. 
For brevity, put $R_m=Q_{\phi(m)}$ and $O_m$ to be the Galois orbit of $R_m$ over $K$.
For any $\sigma \in \mathrm{Gal}(\bar{K}/K)$, we have \[l_v(\sigma P_m) = \mathrm{lcm}_{w\vert v}\{e_w(K(\sigma P_m)\vert K)\} = \mathrm{lcm}_{w\vert v}\{e_{\sigma^{-1} w}(K(P_m)\vert K)\}= l_v(P_m),\] where $w$ runs over all prime ideals of $K(\sigma P_m)$ lying over $v$. 
In particular, Lemma~\ref{int. lemma1} ensures us that $\overline{\mathrm{val}}(\sigma R_m) = \overline{\mathrm{val}}([l_v(P_{\phi(m)})] \sigma P_{\phi(m)})\in S$, and so $\overline{\mathrm{val}}(O_m) \subset S$.

If $X$ has dimension $0$, then it is a point since it is connected, and so the sequence $(R_m)_m$ is constant. 
As $\hat{h}_\mathcal{L}(R_m)\to 0$, we conclude that $\hat{h}_\mathcal{L}(R_m)=0$, that is, $R_m$ is a torsion point, a contradiction. 
In conclusion, $X$ has positive dimension. 
	   
Since the sequence $(R_m)_m$ has no infinite subsequence which is contained in a proper closed subvariety of $X$, the tropical equidistribution theorem \cite[\S 1]{GU10} shows that the sequence of discrete probability measures 
\[ \mu_m = \frac{1}{[K(R_m) : K]} \sum_{x\in O_m} \delta_{\overline{\val}(x)}\] weakly converges to a regular probability measure $\mu$ on $\R^n/\Lambda$ with support equals to $\Trop(X_\nu)$. 
Here, $\delta_{\overline{\val}(x)}$ is the Dirac measure supported at the singleton $\{\overline{\val}(x)\}$. 

On the one hand, $\Trop(X_\nu)$ has cardinality at least two by assumption and is connected according to Subsection \ref{ss-section 2.2}; it is therefore infinite. 
On the other hand, $S$ is a finite set by Lemma \ref{int. lemma1}. 
We can thus find an element $\zeta$ in $\Trop(X_\nu) \backslash S$ as well as a continuous function $f : \R^n/\Lambda \to [0,2]$ taking the value $0$ on $S$ and $1$ at $\zeta$. 
By continuity of $f$, there is a (open) neighbourhood $U$ of $\zeta$ in $\R^n/\Lambda$ such that $f(\zeta')\geq 1/2$ for all $\zeta'\in U$. 
With this choice of $f$, we clearly have $\mu_m(f)=0$ since $\overline{\val}(O_m) \subset S$, while $\mu(f) \geq \int_U f(t) d\mu(t)  \geq \mu(U)/2$. 
We derive to $\mu(U)=0$, a contradiction since any open set in $\R^n/\Lambda$ containing at least one point in $\Trop(X_\nu)$, the support of the measure $\mu$, has positive measure. \qed

\subsection{Proof of Theorem \ref{thm trop}: Full generality} \label{ss-section 3.2}  
We now show the theorem in full generality. 
 This step is quite classic and we closely follow the exposition of \cite[\S 2]{BS04}.

Denote by $A_1,\dots,A_r$ the simple abelian subvarieties of $A_{\overline{K}}$. 
According to Poincar\'e's reducibility theorem \cite[Theorem 8.9.3]{BG06}, $A_{\overline{K}}$ is isogenous to the Cartesian product $B=\prod_{i=1}^r A_i^{e_i}$ for some positive integers $e_i$. 
All these abelian varieties, as well as the isogeny, are defined over a finite Galois extension $K'$ of $K$.  

Let $\phi : B \to A$ and $\psi : A \to B$ be $K'$-isogenies satisfying $\phi \circ \psi = [m]$ on $A$ for some integer $m\geq 1$. 
Clearly, $\phi^*\mathcal{L}$ is a symmetric ample line bundle on $B$ and its restriction to the $i$-th factor $\mathcal{L}_i = \phi^*\mathcal{L} \vert_{A_i}$ is a symmetric ample line bundle on $A_i$. 

Choose a non-torsion point $P\in A(\overline{K})$ and put $\psi(P)=(P_1,\dots, P_r)$. 
At least one of these coordinates is not a torsion point, let us say $P_k$.
The basic properties of N\'eron-Tate heights, see \cite[Chapter 9]{BG06}, give  \[\hat{h}_\mathcal{L}(P) = \frac{1}{m^2} \hat{h}_\mathcal{L}([m]P) = \frac{1}{m^2} \hat{h}_{\phi^*\mathcal{L}}(\psi(P)) = \frac{1}{m^2} \sum_{i=1}^r \hat{h}_{\mathcal{L}_i} (P_i) \geq \frac{\hat{h}_{\mathcal{L}_k}(P_k)}{m^2}. \]
 
 Denote by $v'$ the place of $K'$ lying under $v_{A_k}$. 
 As Theorem \ref{thm trop} is true for geometrically simple abelian varieties, we then get \[ \hat{h}_\mathcal{L}(P) \geq \frac{C/m^2}{\mathrm{lcm}_{w\vert v'}\{e_w(K'(P_k)\vert K')\}^2}, \]
 where $C>0$ is an absolute constant and where $w$ ranges over all places of $K'(P_k)$ lying over $v'$.
 Let $w$ be any place of $K'(P)$. 
As $\psi$ is a $K'$-morphism of varieties, we have $P_k\in A_k(K'(P))$.
The ramification index being multiplicative in towers, we infer that $e_w( K'(P_k) \vert K')$ divides $e_w(K'(P) \vert K')$. 
Next, $K' / K$ is Galois, which implies that $e_w(K'(P) \vert K(P))$ divides $e_w(K'\vert K)$.
It arises from the equality \[e_w(K'(P)\vert K(P))e_w(K(P)\vert K)=e_w(K'(P) \vert K')e_w(K'\vert K)\] that  $e_w(K'(P)\vert K')$ divides $e_w(K(P)\vert K)$. 
In conclusion, $e_w( K'(P_k) \vert K')$ divides $e_w(K(P)\vert K)$. 
The theorem follows since $v'$ is a place lying over $v$, and so\[ \hat{h}_\mathcal{L}(P) \geq \frac{C/m^2}{\mathrm{lcm}_{w\vert v'}\{e_w(K'(P_k)\vert K')\}^2} \geq \frac{C/m^2}{\mathrm{lcm}_{w\vert v}\{e_w(K(P)\vert K)\}^2}. \qed \]

\section{Proof of Theorem \ref{thm Berk}} \label{Section 4}

Recall that each rational function $\phi\in K(X)$ of degree at least $2$  provides an  algebraic dynamical system $(\bbP^1/K, \mathcal{O}(1), \phi)$, and so a Call-Silverman height $\hat{h}_\phi$.

The aim of this section is to provide a large criterion, which ensures us that $\bbP^1(K^{nr,v})$ has the strong Bogomolov property relative to $\hat{h}_\phi$, see Theorem \ref{thm general}.
The proof is done in the second subsection. However, this criterion does not allow us to easily deduce concrete examples. To get it, we will show in the third subsection that the Julia set of a polynomial $\psi\in K[X]$ matches with the set of "maximum points" in the filled Julia set of $\psi$, see Proposition \ref{prop 4.1}. The proof is mainly based on the numerous topological properties of the Berkovich projective line (which are no longer true in higher dimension, which explains why we limited ourselves to the case of the line), which are summarized in the first subsection.  Finally, the fourth (and last) subsection is devoted to the proof of Theorem \ref{thm Berk}.

\subsection{The Berkovich projective line} \label{ss-section 4.1}  
The Berkovich affine line $\A_\Berk^1$ over $\C_v$ is the set of all multiplicative  seminorms on $\C_v[X]$ extending $\vert . \vert_v$ on $\C_v$. 
For example, if $(a,r)\in \C_v \times \R_{\geq 0}$, then \[ \begin{array}{ccccc} 
\zeta_{a,r} & : &  \C_v[X] & \to & \R \\ 
& & P & \mapsto & \underset{z \in D(a,r)}{\mathrm{Sup}} \{ |P(z)|_v \}\\ 
\end{array}\]  
belongs to $\A_\Berk^1$, where $D(a,r)$ is the closed disc in $\C_v$ with radius $r$ and centered at $a$.
We say $\zeta_{a,r}$ is a point of Type I if $r=0$, Type II if $r\in p^{\Q}$ and Type III if $r\in \R_{>0} \backslash p^\Q$. 
We can identify $a$ with the seminorm $\zeta_{a,0}$, and thus see $\C_v$ as a subspace of $\A^1_\Berk$. 

More subtly, if $\mathbf{u}=(D(a_n,r_n))_n$ is a decreasing sequence (for the inclusion) of closed discs with empty intersection (such a sequence exists since $\C_v$ is not spherically complete), then \[ \begin{array}{ccccc} 
\zeta_\mathbf{u}  & : &  \C_v[X] & \to & \R \\ 
& & P & \mapsto & \underset{n\to +\infty}{\lim}  \zeta_{a_n,r_n}(P) \\ 
\end{array}\] describes a new element of $\A^1_\Berk$. 
 Such a seminorm is called a point of Type IV.  
Berkovich's classification theorem claims that $\A^1_\Berk$ is the collection of all points of Type I, II, III or IV. 

The Berkovich topology on $\A^1_\Berk$ is the weakest one for which the maps $\zeta \mapsto \zeta(P)$ are continuous for all $P\in \C_v[X]$. 
The Berkovich affine line is Hausdorff, locally compact, uniquely path-connected, and contains $\C_v$ as a dense subset.
Hence, its one point compactification, called the Berkovich projective line and denoted by $\bbP^1_\Berk$, is Hausdorff, compact, uniquely path-connected and contains $\bbP^1(\C_v)$ as a dense subset.  
The point at infinity is denoted with $\infty$. 
Given $\zeta\in \bbP^1_\Berk$, we put $[\zeta, \infty]$ to be the unique arc from $\zeta$ to $\infty$. 
For example, if $\zeta=\zeta_{a,r}$ is a point of Type I,II or III, then $[\zeta, \infty]=\{\zeta_{a,s}, s\in [r,+\infty]\}$ with the convention that $\zeta_{a,+\infty}=\infty$. 

We now endow $\bbP^1_\Berk$ with a partial order $\preceq$ defined as follows: $\zeta \preceq \zeta'$ if and only if $\zeta'$ lies in $[\zeta, \infty]$. 
The maximum element is the point at infinity and the minimal elements are precisely the points of Type I or IV. 
It can be proven that $\zeta \preceq \zeta'$ if and only if $\zeta(P) \leq \zeta'(P)$ for all polynomials $P\in \C_v[X]$. 
 
The closed Berkovich disc with radius $r\geq 0$ and centered at $a\in \C_v$ is the set \[ \mathcal{D}(a,r)=\{\zeta\in \A^1_\Berk, \zeta \preceq \zeta_{a,r} \}.\] 
It is a connected space  with boundary $\{\zeta_{a,r}\}$. 
As for closed discs in $\C_v$, the intersection of two closed Berkovich discs is either empty or one is contained in the other.
For all the above statements on the Berkovich projective line, see \cite{PT21}. \\

All facts asserted from now, and until the end of this subsection, can be found in \cite[Chapter 10]{BR10}. Given a rational function $\phi\in \C_v(X)$ of degree $d\geq 2$, we can endow $\bbP^1_\Berk$ with a canonical probability measure $\mu_\phi$ and its support $J_\phi$ is called the (Berkovich) Julia set of $\phi$. 
Denote by $O_v$ the ring of integers of $\C_v$ and by $O_v^\times$ the set of units in $O_v$.  
We say that $\phi$ has \textit{good reduction} if $\phi(X)=F_2(X,1)/F_1(X,1)$, where $F_1, F_2 \in O_v[X,Y]$ are two homogeneous polynomials of degree $d$ whose resultant belongs to $O_v^\times$. 
When $\phi$ has good reduction, the probability measure $\mu_\phi$ is the Dirac mass supported at $\zeta_{0,1}$, and so $J_\phi$ is the singleton $\{\zeta_{0,1}\}$. 

This theoretical definition of $J_\phi$ is suitable to obtain our main result, but it does not allow us to exhibit any concrete example, like Theorem \ref{thm Berk}, and so we need a more explicit definition. If $\phi\in\C_v[X]$ is supposed to be a polynomial, then $J_\phi$ is the boundary of the (Berkovich) filled Julia set \[ \mathcal{K}_\phi = \bigcup_{M>0}  \{ \zeta\in \bbP^1_\Berk, \; \; \;  \zeta(\phi^m)\leq M \; \; \; \text{for all} \; \; \; m\geq 1 \} , \] where $\phi^m$ denotes the $m$-fold iteration of $\phi$. 
In other words, $\mathcal{K}_\phi$ is the set of all $\zeta\in\bbP^1_\Berk$ for which the sequence $(\zeta(\phi^m))_m$ stays bounded as $m$ goes to infinity.  
Clearly, $\mathcal{K}_\phi$ is a compact subset of $\bbP^1_\Berk$ not containing $\infty$.    

We will use this definition of $J_\phi$ to prove Theorem \ref{thm Berk}. 
Unfortunately, it does not correspond in full generality with the original if we drop the assumption $\phi\in \C_v[X]$, which limits the applications of Theorem \ref{thm general} below. 
Nevertheless, Riveira-Letellier proved that for all $\phi\in \C_v(X)$ of degree at least $2$, the Julia set of $\phi$ is the closure of the set of repelling points of $\phi$ in $\bbP^1_\Berk$. 
But the authors do not know how to use this in order to provide more concrete examples of Theorem \ref{thm general}. 

\subsection{Statement and proof of the main result} 
We now state our main result.

\begin{thm} \label{thm general}
Let $\phi\in K(X)$ be a rational function of degree at least $2$. 
If there is an element $\zeta\in J_\phi$ such that $\zeta(X)\notin p^{\Z/e_v(K\vert \Q)} \cup \{0\}$, then $\bbP^1(K^{nr,v})$ has the strong Bogomolov property relative to $\hat{h}_\phi$. 
\end{thm}

\begin{proof}
Let $(P_n)_n$ be any sequence of pairwise distinct points in $\bbP^1(K^{nr,v})$ such that $\hat{h}_\phi(P_n)\to 0$. It suffices to get a contradiction in order to prove the theorem (recall that any preperiodic point $P$ satisfies $\hat{h}_\phi(P)=0$). 
By removing at most two terms if needed, we can assume that $P_n\in \bbP^1(K^{nr,v})\backslash (\{0\}\cup \{\infty\})=(K^{nr,v})^\times $ for all $n$.
Yuan's equidistribution Theorem \cite{Y08} tells us that 

\begin{equation} \label{Yuan1}
 \frac{1}{[K(P_n) : K]} \sum_\sigma f(\sigma P_n) \underset{n\to +\infty}{\longrightarrow} \int_{\bbP^1_\Berk} f(t)d\mu_\phi(t)
 \end{equation}
for all continuous functions $f : \bbP^1_\Berk \to \R$, where $\sigma$ runs over all field embeddings from $K(P_n)$ to $\C$ extending the identity on $K$.  

Let $g$ be the real-valued function on $\bbP^1_\Berk$ defined by $\zeta' \mapsto \mathrm{Min}\{\zeta'(X), \zeta'(X)^{-1}\}$. 
It is continuous by definition of the Berkovich topology. 
Write $S$ for the preimage of the closed set $p^{\Z/e_v(K\vert \Q)} \cup \{0\}$ under $g$; it is therefore a closed set, and so compact, in $\bbP^1_\Berk$.
Note that $\zeta\notin S$ by assumption. 

Let $Q\in (K^{nr,v})^\times$ and choose any place $w$ of $K(Q)$ extending $v$. As $v$ is unramified in $K(Q)$, we get, since the ramification index is multiplicative in towers, $e_w(K(Q) \vert \Q) = e_v(K\vert \Q)$. We infer that $\vert Q\vert_w \in p^{\Z/e_v(K\vert\Q)}$; whence \[ \mathrm{Min}\{ \vert Q \vert_w, \vert Q\vert_w^{-1}\} \in p^{\Z/e_v(K\vert\Q)}.\]  
As $\zeta$ extends $\vert .\vert_v$ on $\C_v$, we get $g(Q)=\mathrm{Min}\{ \vert Q \vert_v, \vert Q\vert_v^{-1}\} \in p^{\Z/e_v(K\vert\Q)}$, that is, $Q\in S$. 
The extension $K^{nr,v}/K$ being Galois, it follows that the Galois orbit of $P_n$ over $K$ is included in $S$ for all $n$. 

By Urysohn's lemma, there is a continuous function $f : \bbP^1_\Berk \to [0,2]$ taking the value $0$ on $S$ and $1$ at $\zeta$. 
By continuity of $f$, there is a (open) neighbourhood $U$ of $\zeta$ in $\bbP^1_\Berk$ such that $f(\zeta')\geq 1/2$ for all $\zeta'\in U$. 
With this choice of $f$, the left-hand side in \eqref{Yuan1} is $0$, while the right-hand side is at least $\int_U f(t) d\mu_\phi(t)  \geq \mu_\phi(U)/2$. 
We derive to $\mu_\phi(U)=0$, a contradiction since any open set in $\bbP^1_\Berk$ containing at least one point in $J_\phi$, the support of the measure $\mu_\phi$, has positive measure.  
\end{proof}

\begin{rk} \label{rk 3.1}
\rm{If $\phi$ has good reduction, then $J_\phi=\{\zeta_{0,1}\}$ according to Subsection \ref{ss-section 4.1} and we have  $\zeta_{0,1}(X)=1 \in p^{\Z/e_v(K\vert\Q)}$. 
The condition of Theorem \ref{thm general} is therefore not satisfied. 
This can be compared to the abelian case since Remark \ref{rk 1.10} tells us that $A(K^{nr,v})$ does not have the Bogomolov property relative to the N\'eron-Tate height if $A$ has good reduction at $v$ and if $A(K^{nr,v})$ contains a non-torsion point. }
\end{rk}
 
\subsection{Another formulation for $J_\phi$}
Fix in this subsection a polynomial $\phi\in\C_v[X]$. 
We saw in Subsection \ref{ss-section 4.1} that the Julia set of $\phi$ is the boundary of $\mathcal{K}_\phi$, the filled Julia set of $\phi$. The aim of this subsection is to explicitly compute this boundary. 
This result is probably already known, but it does not seem easy to find any reference.

For $\zeta\in\A^1_\Berk$, we set $\mathcal{D}(\zeta)=\{\zeta'\in \bbP^1_\Berk, \; \zeta'\preceq \zeta\}$. 
It is a closed set in $\bbP^1_\Berk$. 
More precisely, Subsection \ref{ss-section 4.1} shows that $\mathcal{D}(\zeta)=\{\zeta\}$ if $\zeta$ is a point of Type I or IV (because the latter is a minimal element) and that $\mathcal{D}(\zeta)=\mathcal{D}(a,r)$ if $\zeta=\zeta_{a,r}$ is a point of Type II or III. 
As $\mathcal{D}(a,r)$ has boundary $\{\zeta_{a,r}\}$ by Subsection \ref{ss-section 4.1}, we get:

\begin{lmm} \label{lmm 4.1}
For all $\zeta\in\A^1_\Berk$, the boundary of the set $\mathcal{D}(\zeta)$ is the singleton $\{\zeta\}$. 
\end{lmm}

\begin{lmm} \label{lmm 4.2}
Let $\zeta\in \mathcal{K}_\phi$. Then $\mathcal{D}(\zeta)\subset \mathcal{K}_\phi$. 
\end{lmm}

\begin{proof} 
Let $\zeta'\in \mathcal{D}(\zeta)$. 
Subsection \ref{ss-section 4.1} tells us that $\zeta' \preceq \zeta$ implies $\zeta'(P)\leq \zeta(P)$ for all $P\in \C_v[X]$. Taking $P=\phi^m$ proves that the sequence $(\zeta'(\phi^m))_m$ is bounded since $(\zeta(\phi^m))_m$ is bounded by assumption. 
This leads to $\zeta'\in \mathcal{K}_\phi$.  
\end{proof}

Let $\zeta\in \mathcal{K}_\phi$. 
The set $\mathcal{K}_\phi \cap [\zeta, \infty]$ is non-empty, totally ordered since $[\zeta, \infty]$ is, and compact as the intersection of two compact sets in $\bbP^1_\Berk$, which is Hausdorff. 
Hence, it admits a maximum element, say $m_\phi(\zeta)$. 
Write $\mathrm{Max}(\phi)$ for the set of maximum points in $\mathcal{K}_\phi$, that is, the set of $\zeta\in \mathcal{K}_\phi$ for which $m_\phi(\zeta)=\zeta$. 

\begin{lmm} \label{lmm 4.4}
We have $\mathcal{K}_\phi = \bigcup_{\zeta \in \mathrm{Max}(\phi)} \mathcal{D}(\zeta)$. 
\end{lmm}

\begin{proof}
The inclusion $\supset$ arises from Lemma \ref{lmm 4.2} since $\mathrm{Max}(\phi)$ is a subset of $\mathcal{K}_\phi$ by construction. 
Conversely, let $\zeta\in \mathcal{K}_\phi$. 
By definition, $m_\phi(\zeta)$ belongs to $\mathrm{Max}(\phi)$ and we clearly have $\zeta\in\mathcal{D}(m_\phi(\zeta))$ since $\zeta\preceq m_\phi(\zeta)$. 
This shows the other inclusion.
\end{proof}

\begin{lmm} \label{lmm 4.3}
If $\zeta$ and $\zeta'$ are two distinct elements in $\mathrm{Max}(\phi)$, then the sets $\mathcal{D}(\zeta)$ and $\mathcal{D}(\zeta')$ are disjoint.
\end{lmm}

\begin{proof}
Assume that $\zeta$ is a point of Type I or IV. 
Thus, $\mathcal{D}(\zeta)=\{\zeta\}$ and the desired intersection is therefore empty, unless $\zeta\in\mathcal{D}(\zeta')$, that is, $\zeta\preceq \zeta'$.   
As $\zeta$ is a maximal point in $\mathcal{K}_\phi$, and as $\zeta'\in K_\phi \cap [\zeta, \infty]$, it follows that $\zeta=\zeta'$, a contradiction. 
By symmetry, the lemma is also proved when $\zeta'$ is a point of Type I or IV.
We now assume that $\zeta=\zeta_{a,r}$ and $\zeta'=\zeta_{a',r'}$ are points of Type II or III. 
By the foregoing, we have $\mathcal{D}(\zeta)=\mathcal{D}(a,r)$ and $\mathcal{D}(\zeta')=\mathcal{D}(a',r')$. 
By Subsection \ref{ss-section 4.1}, their intersection is empty, unless one contains the other. Suppose that $\mathcal{D}(a,r) \subset \mathcal{D}(a',r')$. 
Again, Subsection \ref{ss-section 4.1} shows that $\zeta=\zeta_{a,r}\preceq \zeta_{a',r'}=\zeta'$ and we get the contradiction as above.
Similarly, we cannot also have $\mathcal{D}(a',r') \subset \mathcal{D}(a,r)$ and the lemma follows. 
\end{proof}

\begin{prop} \label{prop 4.1}
We have $J_\phi =  \mathrm{Max}(\phi)$. 
\end{prop}

\begin{proof} 
 Given a set $S\subset \bbP^1_\Berk$, write $\delta S$ for its boundary. 
It is well-known that the boundary of a disjoint union of closed sets is equal to the disjoint union of boundaries. 
Combining the intermediate lemmas above, we obtain \[ J_\phi = \delta \mathcal{K}_\phi = \delta \left(\bigcup_{\zeta \in \mathrm{Max}(\phi)} \mathcal{D}(\zeta)\right) = \bigcup_{\zeta \in \mathrm{Max}(\phi)} \delta\mathcal{D}(\zeta) = \bigcup_{\zeta \in \mathrm{Max}(\phi)} \{\zeta\}\] and the proposition easily follows.  
\end{proof}

\subsection{Proof of Theorem \ref{thm Berk}} \label{ss-section 4.4}
Let $\phi\in\C_v[X]$ be a polynomial of degree $d$. 
Its leading coefficient is denoted by $\mathrm{lc}(\phi)$. 
The non-archimedean version of the maximum modulus principle claims that \[ \zeta_{a,r}(\phi):= \underset{z\in D(a,r)}{\mathrm{Sup}} \left\{\vert \phi(z)\vert_v \right\} = \underset{n\in\{0,\dots, d\}}{\mathrm{Max}} \left\{ r^n \left\vert \frac{\phi^{(n)}(a)}{n!}\right\vert_v \right\}, \] where $\phi^{(n)}$ is the $n$-th derivative of $\phi$, see \cite[Lemma 21]{KS09}.
We give below an application of this principle, which allows us to delimit the filled Julia set of $\phi$. 

\begin{lmm} \label{lmm 4.5}
Let $\phi\in\C_v[X]$ be a polynomial of degree $d\geq 2$, and let $\zeta_{a,r}$ be a point in $\mathcal{K}_\phi$ with $(a,r)\in\C_v \times \R_{\geq 0}$.
Then $r\leq \vert \mathrm{lc}(\phi) \vert_v ^{-1/(d-1)}$.   
\end{lmm}

\begin{proof}
Let $m$ be a positive integer and recall that $\phi^m$ denotes the $m$-fold iteration of $\phi$. It is a polynomial of degree $d^m$ with leading coefficient \[ \mathrm{lc}(\phi)^{1+d+\dots+d^{m-1}} = \mathrm{lc}(\phi)^{\frac{d^m-1}{d-1}}.\]  
The $d^m$-th derivative of $\phi^m$, which is therefore a constant polynomial, is equal to $(d^m)! \cdot \mathrm{lc}(\phi)^{\frac{d^m-1}{d-1}}$. 
The maximum modulus principle above leads to \[ \zeta_{a,r}(\phi^m) \geq r^{d^m} \vert \mathrm{lc}(\phi)\vert_v^{\frac{d^m-1}{d-1}} = \left(r \vert \mathrm{lc}(\phi)\vert^{\frac{1}{d-1}}_v\right)^{d^m} \vert \mathrm{lc}(\phi)\vert_v^{-\frac{1}{d-1}} . \]  
As the sequence $( \zeta_{a,r}(\phi^m))_m$ is bounded by definition of $\mathcal{K}_\phi$, we immediately infer that $r  \vert \mathrm{lc}(\phi)\vert^{1/(d-1)}_v \leq 1$ since $d\geq 2$. 
The lemma follows. 
\end{proof}

\subsection*{Proof of Theorem \ref{thm Berk}} 
Recall that our assumptions give $\phi(0)\neq 0,$ \[\mu_l\notin \frac{\log p}{e_v(K\vert\Q)}  \Z \; \; \;  \text{and} \; \; \; \mu_l\geq -\frac{\log \vert a_d\vert_v}{d-1}\] for some $l\in\{1,\dots,r\}$.
By \cite[Chapter 6, Theorem 3.1]{C86}, there is a root $a\in \overline{K_v} ^\times$ of $\phi(X)-X \in K_v[X]$ satisfying $-\log \vert a \vert_v=-\mu_l$. 
Thus, $\vert a\vert_v\notin p^{\Z/e_v(K\vert\Q)}$ and $\vert a\vert_v \geq \vert a_d\vert_v^{-1/(d-1)}$.
In addition, $\phi^m(a)=a$ for all positive integers $m$ and it follows that $a\in \mathcal{K}_\phi$. 
By Subsection \ref{ss-section 4.1}, we have $m_\phi(a)=\zeta_{a,r}$ for some real number $r\geq 0$. 
The last lemma asserts that $r\leq \vert a_d\vert_v^{-1/(d-1)}\leq \vert a\vert_v$. 
Next, Proposition \ref{prop 4.1} claims that $\zeta_{a,r}\in J_\phi$. 
Finally, the maximum modulus principle provides  \[ \zeta_{a,r}(X) = \mathrm{Max}\{r, \vert a \vert_v\} = \vert a \vert_v \notin p^{\Z/e_v(K\vert\Q)} \cup \{0\}.\]   
We now finish the proof by applying Theorem \ref{thm general} to $\zeta=\zeta_{a,r}$. \qed

\poubelle{\section{Conflict of interest statement} 

On behalf of all authors, the corresponding author states that there is no conflict of interest. 

\section{Data availability statement}

On behalf of all authors, the corresponding author states that our manuscript has no associated data. 

\section{Ethical Statement}
        On behalf of all authors, the corresponding author states that our manuscript satisfies all conditions for "Ethical Responsibilities of Authors" 
}

\end{document}